\documentclass{gSTA2e}

\usepackage{amsfonts,amssymb,mathtools,amsthm,bm}
\usepackage{graphicx, subcaption}
\usepackage[OT1]{fontenc}
\usepackage[numbers]{natbib}
\usepackage[inline]{enumitem}
\usepackage[colorlinks,citecolor=blue,urlcolor=blue]{hyperref}
\usepackage[dvipsnames]{xcolor}

\theoremstyle{plain}
\newtheorem{theorem}{Theorem}[section]

\newtheorem{proposition}[theorem]{Proposition}

\theoremstyle{definition}
\newtheorem{definition}[theorem]{Definition}
\newtheorem{example}[theorem]{Example}

\theoremstyle{remark}
\newtheorem{remark}[theorem]{Remark}

\def\eqns#1{\begin{equation*}#1\end{equation*}}
\def\eqnl#1#2{\begin{equation}\label{#1}#2\end{equation}}
\def\eqnsa#1{\begin{subequations}\begin{align*}#1\end{align*}\end{subequations}}
\def\eqnmla#1#2{\begin{subequations}\label{#1}\begin{align}#2\end{align}\end{subequations}}

\def\one{\mathbf{1}}
\def\bsc{\bm{c}}
\def\boE{\mathbf{E}}
\def\boF{\mathbf{F}}

\def\bsn{\bm{n}}
\def\bsN{\bm{N}}
\def\boN{\mathbf{N}}

\def\boP{\mathbf{P}}
\def\boX{\mathbf{X}}
\def\bfx{\mathbf{x}}
\def\boY{\mathbf{Y}}
\def\bfy{\mathbf{y}}
\def\bsPi{\bm{\Pi}}	
\def\bsrho{\bm{\rho}}

\def\calB{\mathcal{B}}

\def\calE{\mathcal{E}}
\def\calF{\mathcal{F}}

\def\calN{\mathcal{N}}
\def\calP{\mathcal{P}}
\def\calS{\mathcal{S}}
\def\calT{\mathcal{T}}
\def\calX{\mathcal{X}}
\def\calY{\mathcal{Y}}

\def\bscalX{\bm{\mathcal{X}}}
\def\bscalY{\bm{\mathcal{Y}}}

\def\bbE{\mathbb{E}}

\def\bbN{\mathbb{N}}
\def\bbP{\mathbb{P}}
\def\bbR{\mathbb{R}}

\def\frakF{\mathfrak{F}}
\def\frakM{\mathfrak{M}}
\def\frakN{\mathfrak{N}}
\def\frakY{\mathfrak{Y}}

\def\a{\mathrm{a}}
\def\d{\mathrm{d}}

\DeclareMathOperator{\Sym}{Sym}

\DeclareMathOperator{\dom}{dom}

\DeclareMathOperator{\cov}{cov}

\def\AND{\qquad\mbox{ and }\qquad}
\def\bij{\leftrightarrow}
\def\defeq{\doteq}
\def\ET{,\;\;}
\def\frall#1{(#1)\qquad}
\def\given{\,|\,}
\def\iff{\Leftrightarrow}
\def\implies{\Rightarrow}

\def\ind#1{\one_{#1}}

\def\Liff{\Leftrightarrow}

\def\longbij{\longleftrightarrow}
\def\lwedge{\;\wedge\;}

\def\sfrall#1{(#1)\;\;}
\def\st{\;\,\mbox{s.t.}\;\,}
\def\relpres#1{\stackrel{\sim}{#1}}
\def\inx{\bullet}
\def\eqY{\approx}
\def\SeqY{/\!\eqY}

\begin{document}

\title{On a representation of partially-distinguishable populations}

\author{
\name{Jeremie Houssineau\textsuperscript{a}$^{\ast}$\thanks{$^\ast$Corresponding author. Email: jeremie.houssineau@warwick.ac.uk}
and Daniel E. Clark\textsuperscript{b}}
\affil{\textsuperscript{a}Department of Statistics, University of Warwick;
\textsuperscript{b}D\'epartement Communications, Images et Traitement de l'information, Telecom SudParis}
\received{12 January 2018}
}

\maketitle

\begin{abstract}
A representation of heterogeneous stochastic populations that are composed of sub-populations with different levels of distinguishability is introduced together with an analysis of its properties. It is demonstrated that any instance of this representation where individuals are independent can be related to a point process on the set of probability measures on the individual state space. The introduction of the proposed representation is fully constructive which ensures the meaningfulness of the approach.
\end{abstract}


\begin{classcode}
60A10; 62C10
\end{classcode}

\section{Introduction}

Stochastic populations such as probabilistic multi-object systems are of central importance in diverse applications areas, such as Systems Biology \cite{Chenouard2014}, Robotics \cite{Mullane2011} or Computer Vision \cite{Okuma2004}. In some cases, the sole interest for the practitioner is in their global characteristics, for instance in applications where only the cardinality of the considered population matters as in population dynamics \cite{Hofbauer1998,Turchin2003} or where refined spatial information is not necessarily required. In some other cases, all the individuals of the population can be clearly identified and thus the way the population is represented becomes less important; indeed, problems of this type can be recast into a collection of individual elementary representations. Except in these specific cases, the representation of stochastic populations remains mostly unexplored, in spite of their ubiquity. In general, the population may be only partially distinguishable, i.e.\ some individuals may be identified while some sub-populations might only be described by global non-specific representations such as their cardinality. The objective in this article is to find a natural way of representing these stochastic partially-distinguishable populations. The underlying motivation is that a natural representation should not only be useful in theory when expressing different results and properties, but also in practice when devising approximations for the induced probability distributions. Figure~\ref{fig:distinguishability} shows examples of samples drawn for distributions with different degrees of distinguishability, hinting at the possible drawbacks of using indistinguishable representations for distinguishable populations.

One of the main application areas for the type of representation introduced in this article is in the domain of 
multi-target tracking, which arose from military surveillance radar problems where operators are required to detect, locate, and track multiple potential threats in order to take appropriate countermeasures. Practical solutions to this problem began to emerge in the late 1970s and early 1980s via constructive practical engineering approaches to the problem~\cite{reid1979algorithm,fortmann1980multi}. In the mid-1980s, more fundamental mathematical approaches to the problem were explored, including formulations with point processes~\cite{washburn1987random} and random sets~\cite{mori1984multitarget,mori1986tracking}. Holistic Bayesian approaches were developed further by Portenko~\cite{portenko1997optimal} and Goodman et al.\ \cite{goodman1997mathematics}, and Stone et al.\ \cite{stone1999bayesian}. Up until the 2000s, approaches based on MHT (Multiple Hypothesis Tracking) or JPDA (Joint Probabilistic Data Association) dominated practical applications.  Following this, new approaches based on point processes began to emerge \cite{kreucher2005multitarget}, including solutions showing the equivalence of random set and point process formalisms \cite{vo2005sequential}. More recently, methods for labelled solutions have been developed~\cite{vo2013labeled,papi2015generalized}. One of the limitations of point processes in this context is that they were not originally designed to represent and propagate specific information about individual targets.

As is usual in the target tracking literature, we call a \emph{track} any given set of attributes that allows for characterising a potential individual in the population, e.g.\ a probability distribution on a pre-specified state space together with the sequence of observations that have been assigned to this individual. Since identifying particular individuals is often the practical objective, heuristics are applied to the output of point-process-based algorithms in order to produce tracks. However, since tracks themselves are often not only displayed to the operator but also used for further processing steps, the addition of an ad-hoc step at this stage of the algorithm may prevent them from performing these steps in a holistic and integrated way. For instance, specific data assimilation \cite{Houssineau2015} based on the proposed representation can be easily extended to include classification \cite{Pailhas2016} or sensor management \cite{Delande2014_SensorControl}. Recent applications of the proposed approach by the authors include space situational awareness \cite{Delande2016_Space}, harbour surveillance \cite{Pailhas2016} as well as multi-target tracking from radar data \cite{Houssineau2015_SMC}.

In order to build a natural representation of stochastic populations, it is convenient to start with an idealistic case in which the notion of partial distinguishability can be formalised, and so is done in Section~\ref{sec:generalPopRep}. The concepts and notations introduced in Section~\ref{sec:generalPopRep} are then used as a basis for the introduction of a full representation in Section~\ref{ssec:popRep}. An alternative formulation is finally introduced in Section~\ref{sec:altFormulations}, where simplifications are made in order to make the representation more practical.

\begin{figure}
\centering
\def\svgwidth{.95\textwidth}
\scriptsize
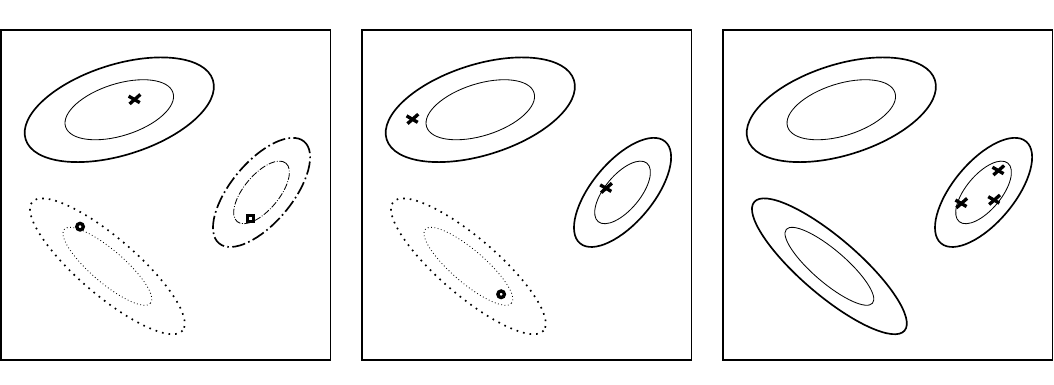
\caption{Distributions (lines) and samples (markers) for different degrees of distinguishability where different markers and line styles are used to emphasize distinguishability: (a) fully distinguishable, (b) partially distinguishable and (c) indistinguishable. The samples in (c) show a drawback of using indistinguishable representations for distinguishable populations, i.e.\ there is no guarantee that the individual samples will come from different modes, as opposed to (a)}
\label{fig:distinguishability}
\end{figure}

Throughout the article, random variables will be implicitly assumed to be defined on the complete probability space $(\Omega,\Sigma,\bbP)$. For any set $A$, denote $\bsPi(A)$ the set of equivalence relations on $A$, and denote $O$ and $I$ the minimal and maximal equivalence relations respectively, i.e.\ $xIy$ is true for any $x,y \in A$ and $xOy$ is false for any $x,y \in A$ such that $x \neq y$.

\section{Describing a population}
\label{sec:generalPopRep}

We consider a \emph{representative set}\footnote{The main notations introduced in this article as well as their meaning are listed in Appendix~\ref{sec:notation}} $\bscalX_{\a}$, i.e.\ a set in which individuals of interest can be uniquely characterised. Because of this characterisation, a population, which can be intuitively understood as a collection of individuals, is formally defined as a subset of $\bscalX_{\a}$. The set $\bscalX$ of all possible populations is then defined as the set of all countable subsets of $\bscalX_{\a}$. In this way, the set $\bscalX$ is itself a representative set for populations. 

An important aspect is that in practice, a more realistic set $\boX$ needs to be considered for the representation of individuals, and $\boX$ is assumed to be a Polish space equipped with its Borel $\sigma$-algebra $\calB(\boX)$. We define $\phi : \bscalX_{\a} \to \boX$ as the projection map relating states in the representative set $\bscalX_{\a}$ with simplified states in $\boX$. Such a simplification is required for most of the applications since the full characterisation of an individual is not usually considered accessible. For instance, the observation may not account for the shape, mass or composition of a given solid, so that only its centre of mass/volume can be inferred; in this case, a point in $\bscalX_{\a}$ might take the form $(\mathbf{s},\mathbf{m},\bfx)$ with $\mathbf{s}$ characterising the shape, $\mathbf{m}$ being the mass and $\bfx$ being the centre of mass, and points in $\boX$ might only describe the latter so that $\phi: (\mathbf{s},\mathbf{m},\bfx) \mapsto \bfx$. 

One of the consequences of the simplified representation described above is that individuals might have the same state in $\boX$. In the context of point process theory \citep{Daley2003}, processes that never have two individuals at the same point are called \emph{simple}. Adopting this term, we aim to propose a representation that does not require \emph{simplicity} in $\boX$ in general. We assume that the set $\boX$ can be written as the union of an Euclidean space $\boX^{\inx}$ and an \emph{isolated point} $\psi$. The latter can be viewed as an \emph{empty state} and is used to provide an image to individuals that cannot be represented on $\boX^{\inx}$ such as individuals that are outside of the zone of interest. A practical example of the meaning of the sets introduced so far is given in Figure~\ref{fig:population}.

\begin{figure}
\centering
\def\svgwidth{.3\textwidth}
\scriptsize
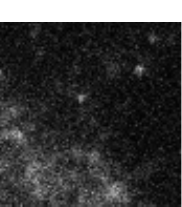
\caption{Image of proteins obtained by Fluorescence Microscopy, where the 3 individual proteins highlighted by small circles constitute an example of population $\calX \in \bscalX$. The set $\boX$ might, for instance, describe only the position of the proteins in the image. For simplicity, the image itself is assumed to characterise the individual proteins by their position and shape and can thus be thought as an element of $\bscalX_{\a}$.}
\label{fig:population}
\end{figure}

The formal definition of indistinguishability considered here differs from the definition of indistinguishable particles in Statistical Mechanics \cite{Tuckerman2010},  though the authors note that analogies between concepts from particle physics for applications in multi-target tracking can be insightful, see e.g.\ \cite{Koch2018}. 
In a multi-particle system, particles of the same species cannot be distinguished by measuring their physical properties such as electric charge or mass since these are exactly equal within a given species; these particles are then said to be indistinguishable. In order to state our definition of indistinguishability, we introduce an observation space $\boY$. In the context of target tracking, the observation space is the space in which the sensor measurements can be expressed, e.g.\ the distance and angle between the target and the sensor for a radar. The following definition introduces what is referred to as strongly indistinguishable individuals in this work.

\begin{definition}
\label{def:strongIndist}
Let $\calX \in \bscalX$ be a set of individuals, let $\boY$ be a given observation space and let $(Y_x)_{x \in \calX}$ be a collection of random variables on $\boY$ characterising the observation process of these individuals, then two individuals $x,x' \in \calX$ are said to be strongly indistinguishable if and only if $Y_x = Y_{x'}$ almost surely.
\end{definition}

Strongly indistinguishable individuals can be related through an equivalence relation $\tau \in \bsPi(\calX)$ defined as follows: $x \tau x'$ holds if and only if the two individuals $x,x' \in \calX$ are strongly indistinguishable. The set
\eqns{
\bscalY \defeq \{(\calX,\tau) \st \calX \in \bscalX \ET \tau \in \bsPi(\calX) \}
}
is introduced in order to represent partially-indistinguishable populations, where $\defeq$ emphasises that the left hand side of the equation is defined as being equal to the right hand side. When individuals are not strongly indistinguishable, they are said to be \emph{weakly distinguishable}. Using the notations of Definition~\ref{def:strongIndist}, two weakly distinguishable individuals $x,x' \in \calX$ are such that $\bbP(Y_x \neq Y_{x'}) > 0$, that is, there is a positive probability to obtain different observations for these individuals.

Even when some individuals are weakly distinguishable, it could happen that the available information is not sufficient to tell them apart. We then say that these individuals are \emph{weakly indistinguishable}. This concept clearly depends on the knowledge about the population and might evolve if additional information is made available. A more formal definition of weak indistinguishability will be given later, but this concept can already be illustrated in the context of Bayesian inference as in the following example.

\begin{example}
\label{ex:distAPriori}
Assuming for simplicity that all the individuals are weakly distinguishable, i.e.\ $\tau = O$, it is often the case in practice that individuals are not distinguished a priori: if $(S_x)_{x \in \calX}$ is a collection of random variables describing the state of the individuals on $\boX$, then for any two individuals $x,x' \in \calX$, it holds that $\bbP(S_x \in B, S_{x'} \in B') = \bbP(S_x \in B', S_{x'} \in B)$ for any measurable subsets $B,B' \in \calB(\boX)$. If a distinct observation $\bfy_x \in \boY$ is received for each individual $x \in \calX$ and if the corresponding likelihood $\ell$ verifies $\ell(\bfy_x \given \cdot) \neq \ell(\bfy_{x'} \given \cdot)$ for any $x,x' \in \calX$ (this can be seen as a form of identifiability) then all the individuals will be distinguished a posteriori.
\end{example}

In multi-target tracking, the main reason for individuals to be strongly indistinguishable is when they share the same state in $\boX$. For instance, the state $\psi$ introduced above can be used to describe any individual that would be outside of the field of view of the sensor collecting observations. These individuals are indeed strongly indistinguishable since they almost surely have no observation associated to them. A typical scenario where individuals might be weakly indistinguishable is when the sensor is likely to collect observations associated with several individuals at once (e.g., because of their close proximity), in this case, there might a be a positive probability that the individuals will eventually be observed separately.

\begin{remark}
Modelling some or all individuals in a population as strongly indistinguishable can also be seen as a simplifying assumption. Indeed, in this case, nothing specific can be learned about these individuals, which can have methodological and computational benefits. This is typically the case in point-process applications where only the population as a whole matters.
\end{remark}

The description of the uncertainty on a given population $\calX \in \bscalX$ can be performed by associating every individual in $\calX$ with a random variable on $\boX$ as in Example~\ref{ex:distAPriori}. This solution, however, does not describe the relation between the different probability distributions related to different individuals, in particular with strongly indistinguishable ones. A global representation of uncertainty is thus sought. One of the most usual ways of describing multiple spatial entities as a whole is given by the theory of point processes. However, this theory is built on the principle there is no interest in characterising specific individuals of the population \citep[p.~124]{Daley2003}. Yet, we wish to model the partially-indistinguishable nature of the individuals in $\calX$ without assuming that they are all strongly indistinguishable, i.e.\ without assuming that $\tau = I$. 

The study of populations composed of indistinguishable individuals is already challenging due to the difficulty in finding a consistent way of describing multiple individuals within a single stochastic object. Examples of questions arising from this issue are: Should the individuals be ordered even though there is no natural way of defining the order? Should the individuals be assumed to be represented at different points of the state space in order to enable a set representation? Should the population be assumed finite in order to proceed to the analysis? There are different ways of answering these questions and each way has to be proved equivalent in some sense to the others \citep{Moyal1962,Macchi1975}. The representation of partially indistinguishable populations raises many additional and equivalently difficult questions. Alternative representations of stochastic populations have to be found in order to tackle this issue.

\section{Representing a population}
\label{ssec:popRep}

Based on the set $\bscalX$ of all possible populations and on the set $\boX$ on which all individuals are represented, we describe a versatile way of introducing randomness in the states of the individuals in $\boX$ which conveys the concept of strong indistinguishability. This is first achieved for a fixed population in Section~\ref{ssec:givenPop} before tackling the full generality of the problem in Section~\ref{ssec:stochasticRepresentation}.

\subsection{For a given population}
\label{ssec:givenPop}

\subsubsection{Construction}

Let $\calY = (\calX,\tau) \in \bscalY$ be a partially-distinguishable population of interest, i.e.\ a set $\calX$ of individuals characterised in $\bscalX_{\a}$ that is equipped with an equivalence relation $\tau$ connecting strongly indistinguishable individuals. The objective is to include the relation between the individuals of $\calX$ in the probabilistic modelling of the population. We first introduce the set
\eqns{
\boF_{\calY} \defeq \big\{ f : \calX \to \boX \st |f^{-1}[\boX^{\inx}]| < \infty \big\}
}
that is composed of mappings $f : \calX \to \boX$ that map finitely many individuals to $\boX^{\inx}$. This condition facilitates the definition of various types of operations on individuals but can be relaxed without inducing major changes in the following results. The set $\calX$ is used as a way of indexing the states in $\boX$ and the actual knowledge of the full individual characteristics $x \in \calX$ is not used. Otherwise, the state of an individual $x \in \calX$ could be directly obtained from the projection $\phi(x) \in \boX$. At the end of this section, we will derive a a way of representing populations that ensures that $\bscalX$ cannot be used to hold information on the state of individuals. In the context of multi-target tracking, it can be convenient to model that infinitely many individuals are at the empty state $\psi$ since, in this case, there is no need to update the number of individuals at this state; see \cite{Houssineau2016_HISP} for more details.

A suitable $\sigma$-algebra of subsets of $\boF_{\calY}$, denoted $\calF^*_{\calY}$ can be introduced as follows: There is a natural topology on $\boF_{\calY}$ that is generated by open sets of the same form as
\eqns{
A = \{ f \st \sfrall{\forall x \in \calX} f(x) \in A_x \},
}
where $\{A_x\}_{x \in \calX}$ is a collection of open sets in $\boX$ that differs from $\{\psi\}$ for finitely many $x \in \calX$ only, i.e.\ $|\{x \in \calX \st A_x \neq \{\psi\} \}| < \infty$. Note that $\{\psi\}$ is indeed open as an isolated point. This topology is denoted $\calT^*_{\calY}$ and $\calF^*_{\calY}$ is defined as the corresponding Borel $\sigma$-algebra. Representations of the population $\calX$ can thus be given as random variables in the measurable space of mappings $(\boF_{\calY},\calF^*_{\calY})$. A random variable $\frakF$ from $(\Omega, \Sigma, \bbP)$ to $(\boF_{\calY},\calF^*_{\calY})$ represents all the individuals in $\calX$ on $\boX$ and is equivalent to a collection of possibly correlated random variables, since indistinguishability has not been taken into account yet.

When two individuals in $\calX$ are strongly indistinguishable, we expect that individual characterisations would not be available; that is, subsets in the $\sigma$-algebra associated with $\boF_{\calY}$ should not allow for evaluating events regarding a specific individual if it is strongly indistinguishable from other individuals. The space $(\boF_{\calY},\calF^*_{\calY})$ is then not fully satisfying as is does not ensure that indistinguishable individuals are well represented. A natural way of circumventing this incomplete representation of the structured population $\calY$ is to make the $\sigma$-algebra $\calF^*_{\calY}$ coarser by ``gluing'' together functions that distinguish indistinguishable individuals.

\begin{example}
\label{ex:gluing}
Suppose that $\calY = (\{x,x'\},I)$, i.e.\ $\calY$ is made of two indistinguishable individuals so that $\calX/\tau = \{ \{x,x'\} \}$. Additionally suppose that $\boX = \boX^{\inx} = \{\bfx,\bfx'\}$, i.e.\ there are only two possible states for the individuals $x$ and $x'$, and assume that $\boX$ is also representative so that $x$ and $x'$ must have different states in $\boX$. There are only $2! = 2$ different distinguishable outcomes $f,g$ in $\boF_{\calY}$ defined by their respective graph as $\{(x, \bfx), (x', \bfx')\}$ and $\{(x, \bfx'), (x', \bfx)\}$. To ensure that the individuals $x,x'$ are indistinguishable, one can glue together these two symmetrical outcomes and define a new set of functions as $\{ \{ f,g \} \}$ (note the additional curly brackets). There is now only one outcome $\{f,g\}$ that does not allow for distinguishing the individuals $x$ and $x'$ as required.
\end{example}

Following Example~\ref{ex:gluing} and denoting $\Sym(\calX,\tau)$ the subgroup of permutations on $\calX$ agreeing with the equivalence relation $\tau$, i.e.\ the ones permuting indistinguishable individuals only, we introduce a binary relation on $\boF_{\calY}$ as follows.

\begin{definition}
\label{def:inducedEqu}
A binary relation $\rho$ on $\boF_{\calY}$ is said to be induced by the equivalence relation $\tau$ if it holds that
\eqnl{eq:inducedEqu}{
\frall{\forall f,f' \in \boF_{\calY}} f\rho f' \iff \exists \sigma \in \Sym(\calX,\tau) ( f = f' \circ \sigma ).
}
\end{definition}

Intuitively, elements of $\boF_{\calY}$ are related through a binary relation whenever they only differ by a permutation of indistinguishable individuals. A representation of the elements of the quotient space $\boF_{\calY}/\rho$ is given in Figure~\ref{fig:mapping}. This binary relation can be proved to have additional properties.

\begin{figure}
    \centering
    \begin{subfigure}[b]{.9\textwidth}
        \centering
        \def\svgwidth{.8\textwidth}
        \scriptsize
        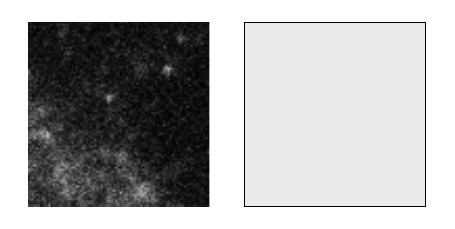
        \caption{Two functions $f$ (orange line) and $f'$ (blue dotted line) in $\boF_{\calY}$ erroneously allowing for the distinction of the two indistinguishable individuals in $\boX$ (circled in green).}
        \label{fig:mapping0}
    \end{subfigure}
    \begin{subfigure}[b]{.9\textwidth}
        \centering
        \def\svgwidth{.8\textwidth}
        \scriptsize
        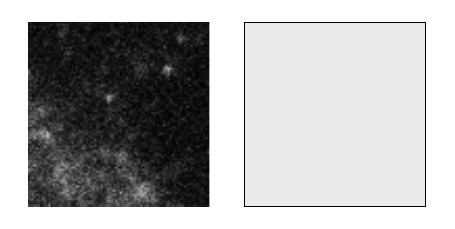
        \caption{Only the set of functions $\{f,f'\}$ is accessible in the quotient set $\boF_{\calY}/\rho$, as represented by the green line between the two indistinguishable individuals and their images.}
        \label{fig:mapping1}
    \end{subfigure} 
    \caption{The two individuals in the bottom left corner are assumed to be indistinguishable, which is incorrectly represented in Figure~\ref{fig:mapping0} and correctly represented in Figure~\ref{fig:mapping1}.}
    \label{fig:mapping}
\end{figure}

\begin{proposition}
\label{prop:def:inducedEqu}
The equivalence relation $\tau$ induces a unique binary relation on $\boF_{\calY}(\boX)$, and this binary relation is an equivalence relation.
\end{proposition}

The proof of Proposition~\ref{prop:def:inducedEqu} relies mostly on the group nature of $\Sym(\calX,\tau)$, as a subgroup of $\Sym(\calX)$. Consequently, only the specific group properties of $\Sym(\calX,\tau)$ will be invoked when proving that the induced binary relation is an equivalence relation.

\begin{proof}
({\it Uniqueness}) Let $\rho$ and $\rho'$ be two binary relations induced by $\tau$. We want to prove that $f\rho f' \Liff f \rho' f'$ holds for any $f,f' \in \boF_{\calY}(\boX)$. Let $\sigma,\sigma'$ be the two permutations in $\Sym(\calX,\tau)$ satisfying \eqref{eq:inducedEqu} for $\rho$ and $\rho'$ respectively. There exists $\sigma''$ in $\Sym(\calX,\tau)$ such that $\sigma \circ \sigma'' = \sigma'$, proving the uniqueness. \\
({\it Reflexivity}) The identity is in $\Sym(\calX,\tau)$. \\
({\it Symmetry}) Existence of an inverse element in $\Sym(\calX,\tau)$. \\
({\it Transitivity}) Closure of $\Sym(\calX,\tau)$.
\end{proof}

Let $\rho$ denote the unique equivalence relation on $\boF_{\calY}$ induced by $\tau$ and let $\xi_{\rho}$ be the quotient map from $\boF_{\calY}$ to $\boF_{\calY}/\rho$ induced by $\rho$. We introduce a $\sigma$-algebra of subsets of $\boF_{\calY}$, denoted $\calF_{\calY}$, which does not allow for distinguishing strongly indistinguishable individuals: Let $\calT_{\calY}$ denote the initial topology on $\boF_{\calY}$ induced by the quotient map $\xi_{\rho}$. We can verify that $\calT_{\calY} \subseteq \calT^*_{\calY}$ holds, meaning that there are fewer open subsets in $\calT_{\calY}$ when compared to $\calT^*_{\calY}$. In other words, we define the topology $\calT_{\calY}$ on $\boF_{\calY}$ as the one inherited from the quotient set $\boF_{\calY}/\rho$. It would also be possible to define random variables directly on $\boF_{\calY}/\rho$ but this would force us to always work with sets of functions instead of working with the functions themselves. Denoting $\calF_{\calY}$ the Borel $\sigma$-algebra induced by $\calT_{\calY}$, we consider random variables from $(\Omega,\Sigma,\bbP)$ to $(\boF_{\calY},\calF_{\calY})$ which characterise subsets of indistinguishable individuals rather than individuals themselves, as required. Note that a reference measure on $(\boF_{\calY},\calF_{\calY})$ can be easily deduced from the reference measure on $\boX$, e.g.\ the Lebesgue measure when $\boX \subseteq \bbR^d$ for some $d > 0$.

\subsubsection{Independence and weak indistinguishability}

Now equipped with suitable spaces for considering the representation of partially-indistinguishable populations, we study the properties of probability measures on $(\boF_{\calY},\calF_{\calY})$. Since populations have an intrinsic multivariate nature, it is natural to introduce a notion of independence for probability measures on $\boF_{\calY}$ as in the following definition, where the variable of integration $\{\bfy_x\}_{x\in\calX}$ is a family of points in $\boX$ indexed by $\calX$ and where $x \mapsto \bfy_x$ is the mapping $f \in \boF_{\calY}$ (i.e.\ from $\calX$ to $\boX$) such that $f(x) = \bfy_x$ for any $x \in \calX$.

\begin{definition}
\label{def:independentIndividuals}
The individuals in $\calX$ are said to be \emph{independent} if the probability measure $P$ on $\boF_{\calY}$ verifies
\eqnl{eq:independentIndividuals}{
P(F) = \int \ind{F}\big(x \mapsto \bfy_x \big) \prod_{x \in \calX} p_x(\d \bfy_x)
}
for any $F \in \calF_{\calY}$, where $\{ p_x \}_{x \in \calX}$ a family of probability measures on $\boX$.
\end{definition}

The expression \eqref{eq:independentIndividuals} of Definition~\ref{def:independentIndividuals} is a convolution of measures based on the operation of creating a function in $\boX^{\calX}$ out of a value in $\boX$ for each the individuals in $\calX$. This notion of independence will be useful as an example of concepts and operations that will be defined in the general case.

The notion of weak indistinguishability that was introduced in Section~\ref{sec:generalPopRep} has not yet been translated into practical terms. As opposed to strongly indistinguishable individuals that are bound through the events in $\calF_{\calY}$, it just happens that there is no specific knowledge about weakly indistinguishable individuals. As a result, weak indistinguishability is a fully probabilistic concept. In order to formally define it, we introduce a mapping $T_{\sigma}$ from $\boF_{\calY}$ into itself for any given $\sigma \in \Sym(\calX)$ defined by
\eqnl{eq:mappingWeakIndist}{
T_{\sigma}: f \mapsto f \circ \sigma.
}
Mappings of this form describe the changes induced by swapping individuals. It is therefore suitable for expressing properties of symmetry for probability measures as in the following definition, where, once again, $\bsPi(\calX)$ is the set of equivalence relations on $\calX$ and where $\xi_*P$ denotes the pushforward of a given probability measure $P$ on a measurable space $(\boE,\calE)$ by a measurable mapping $\xi$ from $(\boE,\calE)$ to another measurable space  $(\boE',\calE')$, i.e.\ $\xi_*P(A) = P(\xi^{-1}(A))$ for any $A \in \calE'$.

\begin{definition}
\label{def:weakIndist}
Let $P$ be a probability measure on $\boF_{\calY}$. The relation of weak indistinguishability induced by $P$ on $\calX$ is defined as
\eqns{
\eta = \sup \big\{\eta' \in \bsPi(\calX) \st \sfrall{\forall \sigma \in \Sym(\calX,\eta')} P = (T_{\sigma})_*P \big\}.
}
\end{definition}

The relation of weak indistinguishability is an equivalence relation by definition. Since $\bsPi(\calX)$ is only a partially ordered set, the greatest element of a given subset might not exist, but it is necessarily unique if it exists. We can show that the relation $\eta$ of weak indistinguishability exists by verifying that any element $\eta' \neq \eta$ in the considered subset can only identify less symmetries than $\eta$. In other words, denoting $\Pi(\eta)$ the partition of $\calX$ induced by $\eta$, there exist at least two subsets in $\Pi(\eta')$ which union is a subset of $\Pi(\eta)$ so that $\Pi(\eta') \leq \Pi(\eta)$ holds for any $\eta'$ in the subset of $\bsPi(\calX)$ of interest. Some of the properties of the relation of weak indistinguishability are given here using the notations of Definition~\ref{def:weakIndist}.

\begin{proposition}
It holds that $\eta \geq \tau$.
\end{proposition}

\begin{proof}
Sets in the $\sigma$-algebra $\calF_{\calY}$ of subsets of $\boF_{\calY}$ do not allow for distinguishing individuals related by $\tau$. Thus, for any given $\sigma \in \Sym(\calX,\tau)$ and $F \in \calF_{\calY}$, it holds that $f \circ \sigma \in F$ for any $f \in F$ so that $P = (T_{\sigma})_*P$ is always true by construction. As a result, the equivalence relation $\tau$ is always in the set of which $\eta$ is the greatest element.
\end{proof}

\begin{example}
Reusing the notations of Definition~\ref{def:weakIndist} and assuming that the individuals in $\calX$ are independent under $P$ and that $\eta$ is the relation of weak indistinguishability induced by $P$, then for any pair $(x,x')$ of individuals in $\calX$, it holds that
\eqns{
(x \eta x') \iff (p_x = p_{x'}).
}
In multi-target tracking applications, it is common to assume that the random variables characterising the targets of interest are conditionally independent (they have independent dynamics and they are independently observed); in this context, targets that have the same distribution are weakly indistinguishable.
\end{example}

The representation of strongly indistinguishable individuals by random variables on $(\boF_{\calY}, \calF_{\calY})$ can be considered as satisfactory. Yet, the true population $\calY$ was supposed to be known so far, even though it is only used as an indexing set, this cannot be assumed in general. It is thus necessary to find a way of dealing with unknown populations.

\subsection{Stochastic representation}
\label{ssec:stochasticRepresentation}

It is natural to reuse the same mechanisms as before to bypass the necessity of knowing the true population when describing it, i.e.\ by defining an appropriate equivalence relation and working on the $\sigma$-algebras induced by the corresponding quotient spaces. However, we will see that the approach that seems the most natural at first does not lead to a satisfactory result. Nonetheless, this approach is detailed below as it motivates the introduction of a more advanced construction.

In order to define formally which features of the elements in $\bscalX$ and $\bscalY$ we want to simplify and which ones we want to retain, we introduce a binary relation on each of these sets as in the following definition.

\begin{definition}
\label{def:equivStrongDist}
Let $\calY,\calY' \in \bscalY$ be two populations equipped with a relation of strong indistinguishability defined via $\calY \defeq (\calX,\tau)$ and $\calY' \defeq (\calX',\tau')$. The binary relations $\sim$ on $\bscalX$ and $\eqY$ on $\bscalY$ are defined as follows
\eqns{
\calX \sim \calX' \Liff |\calX| = |\calX'| \AND \calY \eqY \calY' \Liff \exists \nu : \calY \relpres\longbij \calY',
}
where $\relpres\longbij$ indicates a relation-preserving bijection\footnote{that is, a bijection $\nu$ between $\calX$ and $\calX'$ such that $\nu(x) \tau' \nu(x')$ holds if and only if $x \tau x'$ holds.}.
\end{definition}

It is easy to prove that the binary relations $\sim$ and $\eqY$ on the respective sets $\bscalX$ and $\bscalY$ are equivalence relations. Note that the relation $\sim$ on $\bscalX$ can be equivalently defined as
\eqns{
\calX \sim \calX' \Liff \exists \nu: \calX \bij \calX',
}
where $\bij$ indicates a bijection. This alternative definition highlights the parallel with the equivalence relation $\eqY$ also introduced in Definition~\ref{def:equivStrongDist}. The objective is then to introduce suitable tools to integrate these simplifications at the level of a general population representation.

\subsubsection{Naive attempt}

Since the objective is to extend the results of the previous section to the case where the population $\calY \in \bscalY$ is not known, the first step is to define the union of the sets $\boF_{\calY}$ over all possible populations as
\eqns{
\boF \defeq \bigcup_{\calY \in \bscalY } \boF_{\calY}.
}
Although $\boF_{\calY}$ can be shown to be a Polish space for a given $\calY \in \bscalY$, the set $\boF$ is an uncountable union of (disjoint) Polish spaces so that it is not a Polish space in general. As a consequence, the set $\boF$ is difficult to work with from the viewpoint of Probability Theory. 
In any case, considering a function $f$ in $\boF$ implicitly defines a population via the domain of $f$, so that $\boF$ is not directly useful. As before, we define an equivalence relation on $\boF$ in order to ignore some of the features of its elements.

\begin{definition}
For any $f \in \boF_{\calY}$ and any $f' \in \boF_{\calY'}$, let the binary relation $\bsrho^*$ on $\boF$ be defined as
\eqns{
f \bsrho^* f' \Liff \exists \nu : \calY \relpres\longbij \calY' \big( f = f' \circ \nu \big).
}
\end{definition}

It is also easy to prove that the binary relation $\bsrho^*$ on $\boF$ is an equivalence relation. Equivalence classes in $\boF/\bsrho^*$ do not allow for distinguishing functions with different domains when they have the same co-domain in $\boX$ as required. As before, an appropriate $\sigma$-algebra $\calF$ of subsets of $\boF$ can be deduced from the quotient space $\boF/\bsrho^*$. However, with this approach, the only way of distinguishing equivalence classes in $\calX/\tau$ is through their respective cardinality. In particular, all the functions with a given co-domain in $\boX$ and with a domain of the form $(\calX,O)$ will be equivalent for $\bsrho^*$ so that no individual can be distinguished in spite of the fact that they are all assumed weakly distinguishable. This aspect is illustrated in the following example.

\begin{example}
\label{ex:naive}
Considering, as in Example~\ref{ex:gluing}, a representative set $\boX = \{\bfx,\bfx'\}$ as a state space, assuming that $\bscalX = \{\{x,x'\} \st x,x' \in \bscalX_{\a}, x \neq x' \}$, i.e.\ that populations are made of exactly two individuals, and supposing that individuals are always distinguishable ($\tau = O$), we obtain that
\eqns{
\boF = \{ f: \{x,x'\} \to \{\bfx,\bfx'\} \st x,x' \in \bscalX_{\a} \ET f(x) \neq f(x') \}.
}
We can check that $f\bsrho^*f'$ holds for any $f,f' \in \boF$ (since all the functions in $\boF$ have the same co-domain), so that $\bsrho^* = I$ and $\boF/\bsrho^*$ is a singleton. It follows that the realisations for the individuals $x$ and $x'$ cannot be distinguished from a random variable on $(\boF,\calF)$.
\end{example}

We can still verify that the space $(\boF,\calF)$ is suitable in cases where all the individuals are strongly indistinguishable by showing the relation between the subset
\eqns{
\boF_I \defeq \bigcup_{ \calX \in \bscalX } \boF_{(\calX,I)}
}
of $\boF$ endowed with the $\sigma$-algebra $\calF_I$ induced by $\calF$ and the set $\boN(\boX)$ of integer-valued measures, or counting measures, on $\boX$ equipped with its Borel $\sigma$-algebra $\calN(\boX)$. Such a relation will ensure that random variables on $(\boF_I,\calF_I)$ will be equivalent to point processes on $\boX$ as expected. In the next theorem, $\dom(f)$ will denote the domain of a given function $f$.

\begin{theorem}
\label{thm:eqToPointProcesses}
The mapping $\xi$ defined as
\eqnsa{
\xi : \boF_I & \to \boN(\boX) \\
f & \mapsto \sum_{x \in \dom(f)} \delta_{f(x)},
}
is $\calF_I/\calN(\boX)$-bi-measurable.
\end{theorem}

\begin{proof}
We show that $\xi$ is measurable and then that $\xi[C] \in \calN(\boX)$ for any $C \in \calF_I$: 
\begin{enumerate}[label=\roman*.]
\item \label{proof:eqToPointProcesses:it:preimage} A generating family for the $\sigma$-algebra $\calN(\boX)$ of subsets of $\boN(\boX)$ is found to be made of subsets of the form
\eqns{
C = \{ \mu \in \boN(\boX) \st \mu(B) = i \},
}
for some $B \in \calB(\boX)$ and some $i \in \bbN$. The inverse image of $C$ by the mapping $\xi$ is of the form
\eqns{
\xi^{-1}[C] = \bigg\{ f \in \boF_I \st \sum_{x \in \dom(f)} \ind{B}(f(x)) = i \bigg\}.
}
To verify that $\xi^{-1}[C] \in \calF_I$, we check that
\eqns{
\frall{\forall f \in \xi^{-1}[C]\ET \forall f' \in \boF_I} f \bsrho^* f' \implies f' \in \xi^{-1}[C].
}
By definition we have that
\eqns{
f \bsrho^* f' \Liff \exists \nu : \dom(f) \bij \dom(f') \big( f = f' \circ \nu \big).
}
so that
\eqns{
\sum_{x \in \dom(f)} \ind{B}(f(x)) = \sum_{x \in \dom(f)} \ind{B}(f'(\nu(x))) = \sum_{x \in \dom(f')} \ind{B}(f'(x)) = i,
}
and $f' \in \xi^{-1}[C]$ as required.
\item \label{proof:eqToPointProcesses:it:image} To identify a generating family for the $\sigma$-algebra $\calF_I$, consider a subset of the form
\eqns{
A_{\calX} = \{ f \in \boF_I \st \dom(f) \supseteq \calX \ET \forall x \in \dom(f) (x \in \calX \iff f(x) \in B) \},
}
for some $\calX \in \bscalX$ and a Borel subset $B$ of $\boX$, which includes all the functions based on populations having $\calX$ as a sub-population that maps the individual in $\calX$ into $B$ and all the other individuals outside of $B$. Then, enlarge the subset $A_{\calX}$ by all the functions that are related by $\bsrho^*$ to any function in it, that is
\eqns{
C = \bigcup_{f \in A_{\calX}} [f] = \{ f \in \boF_I \st \exists X \subseteq \dom(f) ( \exists \nu : \calX \bij X ( f \in A_{\nu[\calX]} ))\}
}
which, denoting $i \defeq |\calX|$, can also be expressed as
\eqns{
C = \bigg\{ f \in \boF_I \st \sum_{x \in \dom(f)} \ind{B}(f(x)) = i \bigg\}.
}
It follows easily that
\eqns{
\xi[C] = \{ \mu \in \boN(\boX) \st \mu(B) = i \} \in \calN(\boX).
}
\end{enumerate}
We conclude from \ref{proof:eqToPointProcesses:it:preimage} and \ref{proof:eqToPointProcesses:it:image} that $\xi$ is bi-measurable.
\end{proof}

Theorem~\ref{thm:eqToPointProcesses} shows that a stochastic population where all individuals are strongly indistinguishable is essentially equivalent to a point process; for instance, random variables on $(\boF,\calF)$ with $\boF$ as in Example~\ref{ex:naive} can be seen as equivalent to the counting measure $\delta_{\bfx}+\delta_{\bfx'}$. To obtain the full equivalence would require to define $\xi$ on $\boF_I/\bsrho^*$, in which case it would become an isomorphism. 

The objective is to be able to represent partially-distinguishable populations and therefore events about specific individuals should also be in the $\sigma$-algebra $\calF$. We conclude that random variables on $(\boF,\calF)$ are not suitable in general for our purpose.

\subsubsection{Second attempt}

The approach described in the previous section was incorrectly forcing some weakly distinguishable individuals to be weakly indistinguishable. Since weak indistinguishability is a probabilistic concept, an alternative is to work directly on the set
\eqns{
\boP_{\boF} \defeq \bigcup_{ \calY \in \bscalY } \boP(\boF_{\calY}, \calF_{\calY}),
}
where $\boP(\boE,\calE)$ denotes the set of probability measures on a given measurable space $(\boE, \calE)$. It is then possible to simplify the set $\boP_{\boF}$ while preserving the relations of indistinguishability between individuals. For any $\calY$ and $\calY'$ in $\bscalY$ and any bijection $\nu$ between $\calY$ and $\calY'$, we introduce the mapping $T_{\nu} : \boF_{\calY'} \to \boF_{\calY}$ defined by
\eqns{
T_{\nu}: f \mapsto f \circ \nu.
}
The mapping defined in \eqref{eq:mappingWeakIndist} can be seen as a special case when $\calY = \calY'$.

\begin{definition}
\label{def:equivStrongDist2}
For any populations $\calY,\calY' \in \bscalY$, any $P \in \boP(\boF_{\calY}, \calF_{\calY})$ and any $P' \in \boP(\boF_{\calY'}, \calF_{\calY'})$, let the binary relation $\bsrho$ on $\boP_{\boF}$ be defined as
\eqnl{def:equivStrongDist2:eq:rhoBar}{
P \bsrho P' \Liff \exists \nu : \calY \relpres\longbij \calY' \big( P = (T_{\nu})_*P'\big).
}
\end{definition}

Since each probability measure $P$ in $\boP_{\boF}$ is defined on a single population in~$\bscalY$, the latter can be recovered and will be denoted $\calY_P$ or $(\calX_P,\tau_P)$. If individuals are independent under a given probability measure $P \in \boP_\boF$ then the equivalence class $[P]$ of probability measures related to $P$ via $\bsrho$ is found to be
\eqns{
[P] = \big\{ P' \st \exists \nu : \calY_P \relpres\longbij \calY_{P'} \big( \forall x \in \calX_P ( p_x = p'_{\nu(x)} )  \big) \big\}.
}
This result highlights the structure of the equivalence relation $\bsrho$ and of the mapping $T_{\nu}$ in \eqref{def:equivStrongDist2:eq:rhoBar}.

A given equivalence class in $\boP_{\boF}/\bsrho$ allows for describing the randomness of a population of a given size and structure without knowing the actual population state in $\bscalX$ as required. Such an equivalence class is referred to as a \emph{population representation}. 

Note that the definition of $\bsrho$ does not depend on the relation $\eta$ of weak indistinguishability. Indeed, weak indistinguishability is more an observed property of a population representation rather than a building block that would impose some sort of structure on the mathematical construction of it.

\begin{example}
Considering again the case of Example~\ref{ex:naive} it follows that
\eqns{
\boP_{\boF} = \big\{ P \in \boP\big(\boF_{(\{x,y\},O)}\big) \st x,y \in \bscalX_{\a} \big\}.
}
Focusing on the subset $\boP^*_{\boF}$ of $\boP_{\boF}$ for which individuals are weakly distinguishable and independent, for the sake of simplicity, we find that
\eqns{
P \bsrho P' \iff \exists \nu : \{x,y\} \bij \{x',y'\} \big( (p_x = p'_{\nu(x)}) \lwedge (p_y = p'_{\nu(y)}) \big),
}
for any $P,P' \in \boP^*_{\boF}$, where $\calX_P = \{x,y\}$ and $\calX_{P'} = \{x',y'\}$. In this setup, a point in $\boP^*_{\boF}/\bsrho$, which is in fact an equivalence class, correspond to the configuration where the uncertainty about one individual is described by a given probability distribution $p$ and the uncertainty about the other individual is described by a given probability distribution $p'$, the individuals being weakly indistinguishable if $p = p'$. In other words, individuals are labelled by the probability measures describing the uncertainty about them, these labels being shared by indistinguishable individuals by definition.
\end{example}

The set $\boP_{\boF}/\bsrho$ is not, however, a full answer to the question of the representation of populations since elements of it correspond to a given size and a given structure, i.e.\ a given type of strong indistinguishability. Yet the size and structure of a population are generally unknown and possibly random, and there may be second-order uncertainties on the probability measures in $\boP_{\boF}$ themselves. In general there are many possible distinct configurations for each given cardinality and structure, as illustrated in the following example. 

\begin{example}
In the case of multi-target tracking, the number of targets to be estimated is not generally known in advance, which corresponds to uncertainty on the cardinality of $\calX$ with the considered modelling. For a given cardinality, there are in general several possible configurations based on the assumed \emph{data association}, i.e.\ on the previous observation-to-track decisions. In this context, an element $P$ of $\boP_{\boF}$ corresponds to the distribution of targets on $\boX$ given the number of targets and given the data association, and one way to model the uncertainty on these quantities is to consider a random element of $\boP_{\boF}$.

\end{example}

In order to define random variables on $\boP_{\boF}$, this set also has to be endowed with a suitable $\sigma$-algebra. We follow the same approach as before and consider the initial topology induced by the quotient map of $\bsrho$ and we denote $\calP_{\boF}$ the corresponding Borel $\sigma$-algebra. There is no natural reference measure on $\calP_{\boF}$, but we assume that such a measure is given case by case via a countable subset or a parametric family of probability measures. Similarly, the $\sigma$-algebra on $\bscalY$ is assumed to be induced by the discrete topology on $\bscalY\SeqY$.

A random variable $\frakM$ on $(\boP_{\boF},\calP_{\boF})$ describes all the uncertainties about the system of interest and is referred to as a \emph{stochastic representation}. The interpretation of $\frakM$ can be made easier by separating its law $M \defeq \frakM_*\bbP$ into a marginal and a conditional as 
\eqns{
\frall{\forall B \in \calP_{\boF}} M(B) = \bbE\big[ M(B \given \frakY) \big],
}
where $\frakY$ is the random population induced by $\frakM$ on $\bscalY$, and $M(\cdot \given \frakY)$ is a version of the conditional law of $\frakM$ given $\frakY$, i.e.\ the probability measure representing the second-order uncertainties given the size and structure of the population. This separation of the randomness is straightforward, though it helps to interpret the behaviour of $\frakM$: first a size and a structure $[\calY] \in \bscalY\SeqY$ is randomly selected for the population, then a probability measure on $\boF_{\calY}$ is drawn, where $\calY$ is any element of $[\calY]$, describing the uncertainty about the considered type of population and ensuring that there is no specific knowledge about strongly indistinguishable individuals. Which population has been chosen from $[\calY]$ is irrelevant since the mapping $\calY \mapsto M(B \given \frakY = \calY)$ has to be measurable for any $B \in \calP_{\boF}$.

\begin{remark}
The probability $\bbP\big(\frakY \in [(\calX,\tau)] \big)$ only depends on the size of $\calX$ and on the size of the subsets in $\calX/\tau$. For instance, we can evaluate the probability for a realisation $\calY$ of $\frakY$ to contain exactly 3 strongly indistinguishable individuals and 2 weakly distinguishable ones, however, we cannot assess the probability of any event regarding the states of these individuals in $\bscalX_{\a}$. 
\end{remark}

\subsection{Statistics}

The stochastic representation $\frakM$ is a random element of $\boP_{\boF}$ inducing a random size and a random structure via the randomly selected probability distributions. Because the realisations of $\frakM$ are probability measures on different spaces, these realisations are not directly summable, yet statistics for some aspects of $\frakM$ can be defined. In order to assess events regarding the state of some or all of the individuals in a population, it is convenient to have appropriate mappings that can simplify the structure of the elements of $\boF_{\calY}$.

\begin{proposition}
\label{res:mappingMoments}
Let $\calY = (\calX,\tau) \in \bscalY$ and let $X$ be a subset of $\calX$, then the mapping $t_X$ defined as
\eqnmla{eq:mappingMoments}{
t_X : \boF_{\calY} & \to \boN(\boX) \\
f & \mapsto \sum_{x \in X} \delta_{f(x)},
}
is $\calF_{\calY}/\calN(\boX)$-measurable if and only if $X$ is the union of elements of $\calX/\tau$.
\end{proposition}

Proposition~\ref{res:mappingMoments} allows for studying the law $(t_X)_* P$ for any $P \in \boP(\boF_{\calY},\calF_{\calY})$, $\calY \in \bscalY$, which is a point-process distribution.

\begin{proof}
As mentioned before, the $\sigma$-algebra $\calN(\boX)$ is generated by subsets of the form $C = \{ \mu \in \boN(\boX) \st \mu(B) = i\}$, for some $B \in \calB(\boX)$ and some $i \in \bbN$, and it holds that
\eqns{
t_X^{-1}[C] = \bigg\{ f \in \boF_{\calY} \st \sum_{x \in X} \ind{B}(f(x)) = i \bigg\}.
}
The mapping $t_X$ is measurable if and only if
\eqns{
\frall{\forall f \in t_X^{-1}[C]\ET \forall f' \in \boF_{\calY}} f \rho f' \implies f' \in t_X^{-1}[C],
}
which is equivalent to
\eqns{
\frall{\forall f \in t_X^{-1}[C]\ET \forall \sigma \in \Sym(\calX,\tau)} \sum_{x \in X} \ind{B}(f(\sigma(x))) = i.
}
This last statement holds if and only if $\sigma[X] = X$ for all $\sigma \in \Sym(\calX,\tau)$, i.e.\ if and only if there exists partition of $\calX$ containing $X$ and being coarser than~$\calX/\tau$. The condition on $X$ in the proof of the proposition follows easily.
\end{proof}

We consider a few increasingly sophisticated examples here:
\begin{itemize}
\item The expected number of individuals and the expected number of strongly indistinguishable sub-populations are respectively $\bbE[ N(\frakM) ]$ and $\bbE[ \tilde{N}(\frakM) ]$ with, for any $\calY = (\calX,\tau) \in \bscalY$ and any $P \in \boP(\boF_{\calY},\calF_{\calY})$,
\eqns{
N(P) = |\calX| \AND \tilde{N}(P) = |\calX/\tau|.
}
\item The expected number of individuals within a subset $B \in \calB(\boX)$ is $\bbE[ N_B(\frakM) ]$ with
\eqns{
N_B(P) = \bbE[t_{\calX}(\frakF)(B)]
}
where $\frakF$ is any random variable with distribution $P$.
\item Assuming that all individuals represented by $\frakM$ are weakly distinguishable, the expected number of individuals with marginal law within a subset $C$ of $\boP(\boX, \calB(\boX))$, henceforth denoted $\boP(\boX)$ for compactness, is $\bbE[L_C(\frakM)]$ with, for any $\calY = (\calX,\tau) \in \bscalY$ and any distribution $P \in \boP(\boF_{\calY},\calF_{\calY})$,
\eqns{
L_C(P) = \sum_{x\in\calX} \ind{C}\big((\hat{t}_x)_*P\big)
}
where the mapping
\eqnsa{
\hat{t}_x : \boF_{\calY} & \to \boX \\
f & \mapsto f(x)
}
defined for any $x \in \calX$, can be easily proved to be $\calF_{\calY}/\calB(\boX)$-measurable (it can be seen as a special case of \eqref{eq:mappingMoments} with $X = \{x\}$ and $f(x)$ being used directly instead of $\delta_{f(x)}$). The quantity $\bbE[L_C(\frakM)]$ can be useful in practice, e.g.\ to compute the expected number of individuals 
\begin{enumerate*}
\item who are expected to be within some given subset of $\boX$,
\item who are more than $\alpha\%$ likely to be within some given subset of $\boX$ for a fixed $\alpha$, or
\item whose marginal law has a second moment/an entropy that is upper bounded by some fixed constant.
\end{enumerate*}
\end{itemize}

Many applications are concerned with the study of populations where the individuals are independent. The simplifications induced by such an assumption are important enough to justify studying this case specifically, and so is done in the next section.

\section{Alternative formulation}
\label{sec:altFormulations}

The objective now is to show that the problem can be formulated on more standard sets than $\boP_{\boF}$. We focus on one alternative formulation which relies on integer-valued measures, however, other formulations are possible, e.g.\ with product measures on suitably defined spaces. These types of formulation already exist for point processes as described in \cite{Moyal1962} and \cite{Ito2013}. The following assumption will henceforth be considered:
\begin{enumerate}[label=\bfseries A.\arabic*,series=hyp]
\item \label{hyp:independence} Individuals are independent.
\end{enumerate}
The subset of $\boP_{\boF}$ composed of probability measures for which all individuals are independent is denoted $\boP^*_{\boF}$ and is equipped with the $\sigma$-algebra $\calP^*_{\boF}$ induced by $\calP_{\boF}$. For a given $P \in \boP^*_{\boF}$, we denote $\calX_P \in \bscalX$ the population on which $P$ is based and $\{p_x\}_{x \in \calX_P}$ the corresponding family of individual probability distributions on $\boX$.

One of the most direct alternative formulations uses the concept of integer-valued measures or counting measures. A connection between the specific notion of population representation and the more common concept of counting measure is established in the following proposition. Since $\boP(\boX)$ is a Polish space when equipped with the topology induced by the Prokhorov metric \cite{Prokhorov1956}, the set $\boN(\boP(\boX))$ can also be made Polish \cite{Daley2003} and is therefore equipped with its Borel $\sigma$-algebra denoted $\calN(\boP(\boX))$. Also, the Borel $\sigma$-algebra of $\boP(\boX)$ is denoted by $\calP(\boX)$.

\begin{theorem}
\label{thm:popProcToMeasure}
The mapping  $\zeta : \boP^*_{\boF} \to \boN(\boP(\boX))$, defined as
\eqnl{eq:prop:popProcToMeasure}{
\zeta : P \mapsto \sum_{x \in \calX_P} \delta_{ p_x },
}
is $\calP^*_{\boF}/\calN(\boP(\boX))$-measurable.
\end{theorem}

\begin{proof}
The Borel $\sigma$-algebra on $\boN(\boP(\boX))$ is the one generated by subsets of the form
\eqns{
C = \{ \mu \in \boN(\boP(\boX)) \st \mu(B) = i \},
}
for some $B \in \calB(\boP(\boX))$ and $i \in \bbN$. The inverse image of $C$ by $\zeta$ is found to be
\eqns{
\zeta^{-1}[C] = \bigg\{ P \in \boP_{\boF} \st \sum_{x \in \calX_P} \ind{B}(p_x) = i \bigg\},
}
where $\calX_P$ is the population on which $P$ is defined and $\{p_x\}_{x \in \calX_P}$ is the indexed family of probability measures on $\boX$ induced by $P$. Following the same route as in the proof of Theorem~\ref{thm:eqToPointProcesses}, we can verify that $\zeta^{-1}[C] \in \calP^*_{\boF}$.
\end{proof}

Theorem~\ref{thm:popProcToMeasure} shows that stochastic representations can be expressed as a random counting measure, or point process, on the set of probability measures on $\boX$. The transformation $\zeta$ introduced in this proposition does not preserve the representation of strong indistinguishability and is not bi-measurable as a consequence. This can be seen as beneficial in practice since the observability of strong indistinguishability is often not realistic. The only individuals that are known to be strongly indistinguishable in this case are the ones that are almost surely at the same point of the state space. 

\begin{remark}
It is possible to relax Assumption~\ref{hyp:independence} to: individuals that are not strongly indistinguishable are independent. In this case, the corresponding subset of stochastic representations could be mapped to $\boN(\boP(\boN(\boX), \calN(\boX)))$, that is, to the set of point processes on the space of probability distributions of point processes (the latter characterising sub-populations of strongly-indistinguishable individuals). In this configuration, the relation of strong indistinguishability can be preserved, but at the expense of a more complex set of counting measures.
\end{remark}

As a point process on $\boP(\boX)$, i.e.\ as a random variable on $(\boN(\boP(\boX)), \calN(\boP(\boX)))$, $\frakM$ can be characterised by its probability-generating functional (p.g.fl.) $G$, defined for any non-negative bounded measurable function $h$ on $\boP(\boX)$ as \citep{Daley2008}
\eqnsa{
G(h) & \defeq \bbE\bigg[ \exp \bigg( \int \log h(p) \frakM(\d p)  \bigg) \bigg] \\
& = c(0) + \sum_{n \geq 1} c(n) \int \prod_{i=1}^n h(p_i) P_n(\d(p_1,\dots,p_n)),
}
where $c \in \boP(\bbN)$ is defined as $c(n) = \bbP(\frakM(\boP(\boX)) = n)$ for any $n \geq 0$ and $P_n$ is the distribution of $\frakM$ on $\boP(\boX)^n$ conditioned on $\frakM(\boP(\boX)) = n$, for any $n > 0$.

The simplicity of this integer-valued measure formulation comes from the fact that the state space does not actually appear in the equations, thereby allowing for more flexibility in the expressed quantity.

\subsection{Parametrised family of probability measures}

A special case of interest is found when the support of the considered stochastic representations is within a family of probability measures parametrised by a set $\Theta \subseteq \bbR^{d_{\Theta}}$ for some $d_{\Theta} > 0$. This enables some of the properties of stochastic representations to be studied on the simpler set $\Theta$. The following additional assumption is henceforth considered:
\begin{enumerate}[resume*=hyp]
\item \label{hyp:parametrisedFamily} Stochastic representations take values in a parametrised family of probability measures.
\end{enumerate}
Under Assumption~\ref{hyp:parametrisedFamily}, let $\calS_{\Theta} = \{p_{\theta}\}_{\theta \in \Theta}$ be an identifiable family of probability measures on $\boX$ encompassing the support of $\frakM$. In this context, identifiability means that $p_{\theta} \neq p_{\theta'}$ whenever the parameters $\theta,\theta' \in \Theta$ are different. The point process $\frakM$ induces a point process $\frakN$ on $\Theta$ in the following way:
\eqns{
\frakN(B) = \frakM(T(B))
}
for any $B \in \calB(\Theta)$, where $T : \Theta \ni \theta \mapsto p_{\theta} \in \boP(\boX)$ is assumed to be bi-measurable. Straightforwardly, any point process $\frakN'$ on $\Theta$ induces a point process on $\boP(\boX)$ defined as $T_*\frakN'$. One of the consequences on this relation is the ability to recover statistics for $\frakM$ from the ones for $\frakN$, for instance the expected number of individual laws within the measurable subset $C$ of $\boP(\boX)$ can be recovered via
\eqns{
\bbE[\frakM(C)] = \bbE[\frakN(T^{-1}(C))].
}
The p.g.fl.\ of $\frakM$ can now be equivalently expressed as
\eqnsa{
G(h) & = \bbE\bigg[ \exp \bigg( \int \log h(p_{\theta}) \frakN(\d \theta)  \bigg) \bigg] \\
& = c(0) + \sum_{n \geq 1} c(n) \int \prod_{i=1}^n h(p_{\theta_i}) Q_n(\d(\theta_1,\dots,\theta_n)),
}
where $Q_n$ is the distribution of $\frakN$ on $\Theta^n$ conditioned on $\frakN(\Theta) = n$, for any $n > 0$. If the population under consideration is fully distinguishable almost surely then the point process $\frakM$ is simple and $Q$ admits a density w.r.t.\ the Lebesgue measure on $\Theta$.

\subsection{Discrete set of probability measures}

We also formulate an assumption that is of interest when devising practical estimation algorithms:
\begin{enumerate}[resume*=hyp]
\item \label{hyp:countableSupport} The set $\Theta$ is countable.
\end{enumerate}
As a consequence of Assumption~\ref{hyp:countableSupport}, the point process $\frakN$ induced by $\frakM$ is equivalent to a random variable $\bsN$ on the set $\bar\bbN^{\Theta}$, with $\bar\bbN = \bbN \cup \{+\infty\}$. Then $\frakM$ can be expressed as
\eqns{
\frakM = \sum_{\theta \in \Theta} \bsN_{\theta} \delta_{p_{\theta}}.
}
Note that $\bsN$ verifies $\bsN_{\theta} < \infty$ for any $\theta \in \Theta$ such that $p_{\theta} \neq \delta_{\psi}$. A realisation $\mu$ of $\frakM$ can be denoted $\mu_{\bsn}$ with $\bsn$ the corresponding realisation of $\bsN$ in order to underline the multiplicity of each atom in $\calS_{\Theta}$. The law $P$ of $\frakM$ on $\boN(\boP(\boX))$ can then be expressed as
\eqns{
P(B) = \int \ind{B}(\mu_{\bsn}) \bsc(\d\bsn)
}
for any Borel subset $B$ of $\boN(\boP(\boX))$, where $\bsc$ is the induced probability measure on~$\bar\bbN^{\Theta}$. In the context of multi-target tracking, each sequence of observation yields a potential individual distribution so that $\Theta$ would be defined as the set of all sequences of past observations.

\begin{example}
If a population is known to contain exactly $3$ individuals and if the only available probability distributions for these individuals are the ones in the set $\calS_{\Theta} = \{p_1,p_2\}$, in which case $\Theta = \{1,2\}$ and elements of $\bbN^{\Theta}$ can be seen as pairs of integers, then the population representation can be any of the following:
\eqns{
\mu_{3,0} = 3\delta_{p_1}, \quad \mu_{2,1} = 2\delta_{p_1}+\delta_{p_2}, \quad \mu_{1,2} = \delta_{p_1} + 2\delta_{p_2}, \quad \mu_{0,3} = 3\delta_{p_2}.
}
For instance, $\mu_{2,1}$ describes the case where the uncertainty about two of the individuals is described by $p_1$, so that these two individuals are indistinguishable, and the uncertainty about the other individual is described by $p_2$. In this form it is not known whether the two weakly indistinguishable individuals are also strongly indistinguishable or not. 
\end{example}

Identifying a countable family $\calS_{\Theta}$ of probability measure and additionally assuming that $\bsN_{\theta} < \infty$ even if $p_{\theta} = \delta_{\psi}$ enables a simplification of the expression of the p.g.fl.\ of $\frakM$ to
\eqns{
G(h) = \bbE\bigg[ \prod_{\theta \in \Theta} h(p_{\theta})^{\bsN_{\theta}} \bigg]  = \sum_{\bsn \in \bbN^{\Theta}} \bsc(\bsn) \prod_{\theta \in \Theta} h(p_{\theta})^{\bsn_{\theta}},
}
which is related to the probability-generating functional of $\frakN$ as expected. For instance, if $\Theta = \{1,\dots,k\}$, then $G(h) = G'(h(p_1),\dots,h(p_k))$, with $G'$ the probability-generating function of $\bsN$ defined as
\eqns{
G'(z_1,\dots,z_k) \defeq \sum_{\bsn \in \bbN^{\Theta}} \bsc(\bsn) z_1^{\bsn_1} \dots z_k^{\bsn_k}.
}
Assumption~\ref{hyp:countableSupport} also yields a simpler expression of the statistics induced by a stochastic representation $\frakM$ on $\boN(\boP(\boX))$. Of particular interest are the mean $M(B)$ and variance $V(B)$ for the number of individual laws within a measurable subset $C$ of $\boP(\boX)$, characterised by
\eqnsa{
M(C) & \defeq \bbE[ \frakM(C) ], \\
V(C) & \defeq \bbE[ \frakM(C)^2 ] - M(C)^2,
}
whenever they exist. These quantities are well defined since $\frakM$ is a random measure. If the quantities of interest are the mean and variance on the state space $\boX$, then the mapping
\eqnsa{
\Phi_B : \boP(\boX) & \to \bbR \\
p & \mapsto p(B),
}
can be introduced for any $B \in \calB(\boX)$ and is $\calP(\boX)/\calB(\bbR)$-measurable by \cite[Proposition~A2.5.IV]{Daley2003}. The \emph{collapsed} first moment $M'(B)$ and variance $V'(B)$, describing the number of individuals within $B \in \calB(\boX)$ can then be defined as \cite{Delande2016_DISP}
\eqnsa{
M'(B) & \defeq \bbE[ \frakM(\Phi_B) ] = \sum_{\theta \in \Theta} m_{\theta} p_{\theta}(B), \\
V'(B) & \defeq \bbE[ \frakM(\Phi_B)^2 ] - M'(B)^2 = \sum_{\theta,\theta' \in \Theta} \cov_{\theta,\theta'} p_{\theta}(B) p_{\theta'}(B).
}
where $\frakM(\Phi_B) \defeq \int \Phi_B(p) \frakM(\d p)$, $m_{\theta} \defeq \bbE[\bsN_{\theta}]$ and $\cov_{\theta,\theta'} \defeq \bbE[\bsN_{\theta} \bsN_{\theta'}] - \bbE[\bsN_{\theta}] \bbE[\bsN_{\theta'}]$. These relations between $\bsN$ and $\frakM$ are connected to the relation between the p.g.fl.\ $G$ and the probability-generating function $G'$. Even in the simple configuration induced by Assumption~\ref{hyp:countableSupport}, the structure of the proposed representation of stochastic populations enables more diverse types of statistics to be computed when compared to point processes on the state space, which is practically relevant for describing filtering algorithms for multi-object dynamical systems \cite{Delande2014_Var}.

\section*{Conclusion}

Starting from general considerations about the concepts of individual and population and about the partially-indistinguishable knowledge that may be available about them, we presented increasingly general notions in an attempt to faithfully describe the multi-faceted nature of the corresponding uncertainties. After a suitable level of generality was reached, an alternative way of expressing the uncertainty about these complex systems has been introduced. This alternative expression highlights the nature of the proposed representation by identifying it with a point process on the set of probability measures on the individual state space, under the assumption of independence between individuals. Future work includes the study of algorithms based on the introduced representation of populations as well as their theoretical analysis following, for instance, the approach of \cite{DelMoral2013,DelMoral2015}.

\newpage

\appendix

\section{Notation}
\label{sec:notation}

\vspace{-1em}
\small
\begin{itemize}[leftmargin=.12\columnwidth]
\item[$\bscalX_{\a}$:] \emph{Representative set}, i.e.\ set in which individuals are uniquely characterised
\item[$\bscalX$:] Set of all possible populations, i.e.\ set of all countable subsets of $\bscalX_{\a}$
\item[$\boX$:] State space, defined as the union of an Euclidean space $\boX^{\inx}$ and an isolated point~$\psi$
\item[$\bscalY$:] Set of \emph{structured populations}, i.e.\ populations equipped with a relation of strong indistinguishability
\item[$\calX$, $\calY$:] A given population (resp.\ structured population), i.e.\ an element of $\bscalX$ (resp.\ $\bscalY$)
\item[$x$, $\bfx$:] Elements of $\bscalX_{\a}$ and $\boX$ respectively
\item[$\tau, \eta$:] Equivalence relations of strong (resp.\ weak) indistinguishability
\item[$\boP(E,\calE)$:] Set of probability measures on the measurable space $(E,\calE)$
\item[$(\boF_{\calY},\calF_{\calY})$:] Measurable space of functions from $\calY$ to $\boX$ that do not allow for distinguishing strongly indistinguishable individuals
\item[$\boP_{\boF}$:] Union over all structured populations $\calY \in \bscalY$ of $\boP(\boF_{\calY},\calF_{\calY})$
\item[$\calP_{\boF}$:] $\sigma$-algebra on $\boP_{\boF}$ that do not allow for accessing the state of individuals in $\bscalX_{\a}$
\item[$\frakM$:] A \emph{stochastic representation}, i.e.\ a random variable on $(\boP_{\boF},\calP_{\boF})$ or a point process/random counting measure on $\boP(\boX)$
\end{itemize}

\bibliography{Thesis}

\end{document}